\documentclass[journal]{IEEEtran}

\usepackage{amsmath}    
\usepackage{graphics,epsfig,subfigure}    
\usepackage{verbatim}   
\usepackage{color}      
\usepackage{subfigure}  
\usepackage{hyperref}   
\usepackage{amssymb}
\usepackage{wasysym}
\usepackage{epstopdf}
\usepackage{graphicx}
\usepackage{mathrsfs}

\newcommand{\expect}[1]{\mathbb{E}\big\{#1\big\}}

\newcommand{\script}[1]{{{\cal{#1} }}}

\newtheorem{lemma}{\textbf{Lemma}}

\newtheorem{theorem}{\textbf{Theorem}}

\allowdisplaybreaks
\begin{document}

\title{Optimal Demand Response with Energy Storage Management}
\author{\large{Longbo Huang, Jean Walrand, Kannan Ramchandran}  
\thanks{Longbo Huang  (http://www.eecs.berkeley.edu/$\sim$huang), Jean Walrand, and Kannan Ramchandran are with the EECS department of University of California, Berkeley, CA, 94720.  
} 
\markboth{Draft}{Huang}
}

\maketitle
\thispagestyle{empty}
\pagestyle{empty}

\begin{abstract}
In this paper, we consider the problem of optimal demand response and energy storage management for a power consuming entity. The entity's objective is to find an optimal control policy for deciding how much load to consume, how much power to purchase from/sell to the power grid, and how to use the finite capacity energy storage device and renewable energy, to minimize his average cost, being the disutility due to load-shedding and cost for purchasing power. Due to the coupling effect of the finite size energy storage, such problems are challenging and are typically tackled using dynamic programming, which is often complex in computation and requires substantial statistical information of the system dynamics. We instead develop a low-complexity algorithm called Demand Response with Energy Storage Management \textsf{(DR-ESM)}. \textsf{DR-ESM} does \emph{not} require any statistical knowledge of the system dynamics, including the renewable energy and the power prices. It \emph{only requires the entity  to solve a small convex optimization program with $6$ variables and $6$  linear constraints} every time for decision making. We prove that \textsf{DR-ESM} is able to achieve near-optimal performance and explicitly compute the required energy storage size. 
\end{abstract}

\vspace{-.2in}
\section{Introduction}
The increasing penetration of renewable energy and distributed generation inevitably increase the uncertainty in the smart grid.
Energy storage reduces the impact of these uncertainties by smoothing out fluctuations and reducing the mismatch between supply and demand  \cite{eilyan-storage-acc11}, \cite{pravin-risk11}. 
Under such circumstances, in order to guarantee reliable delivery of electricity to consumers,  and to ensure stability of the transmission and distribution systems, energy storage technology and demand response schemes are both being  integrated into the power grid. 
Therefore, it is essential to design control schemes that best utilize both approaches.  

In this paper, we consider the problem of optimal energy management for power  consumers with energy storages. Specifically, we consider a power consuming entity, e.g., a group of households or a commercial building, who is equipped with a finite energy storage device, and needs to meet his power demand by using his renewable energy, stored power, and purchased power from the grid. The energy consuming entity pays the grid for drawing power from it, but can also sell power to the grid to compensate some of the costs. The objective of the  energy consuming entity is to find a control policy for deciding how much load to consume, when and how much power to draw from the grid, how to charge the energy storage, and how to sell power back to the grid, so as to minimize its time average cost, i.e., disutility due to load shedding and payment to the grid. 

This problem is a challenging problem. The finite capacity of the energy storage device couples all the storage charging/discharging actions across time. The power prices for purchasing power from the grid is time-varying, which results in time-varying costs. Also, the available renewable energy is stochastic and may be difficult to forecast long beforehand. Moreover, the power consuming entity's disutility due to power consumption may also vary over time due to the changing system environment, e.g., temperature. Finally, the entity's ability to carry out demand response further complicates the problem. 

There have been many previous work on developing optimal control schemes for optimally utilizing energy storage systems and demand response in the smart grid. 
\cite{residential-storage11}, \cite{elgamal-storage11}, \cite{koutsopoulos-11} formulate the problem of storage management using dynamic programming (DP) and derive threshold-based control policies. 
\cite{boyd-storage11} develops energy storage control policies using receding horizon control. \cite{opf-storage-chandy}, \cite{opf-dist-storage-gayme} formulate the problem of finite horizon storage management as convex programs.  
\cite{distributed-sotrage} uses quadratic control techniques to study the scaling effect of energy storage in the power grid. 
\cite{nali-dr-11} considers the problem of demand response with energy storage in a finite horizon, and formulates the problem as a convex optimization program. 
\cite{libin-cdc-11} and \cite{hwr-dr-2012} develop optimal multi-stage power procurement and demand response schemes that do not include storage. 
However, the aforementioned works either assume that the system is static, or require complete statistical knowledge of all the system dynamics, or ignore the physical energy storage capacity constraints, or do not include energy storage. 
%
Three recent works \cite{control-rechargeable-twc10},  \cite{rahulneely-storage}, \cite{huangneely-energy-mobihoc} construct low-complexity algorithms for energy management based on the Lyapunov optimization technique  \cite{neelynowbook}, and explicitly consider the actual energy storage control. 
However, in these works, the energy storage devices are used mainly as a tool for storing and supplying power for the consumers. Thus, they did not investigate the potential economic aspects of the energy storage, i.e., selling power back to the grid, as well as demand response.

In this paper, we consider both demand response and energy storage management. We explicitly take into account the fact that the energy storage has finite capacity and the system environment can be time-varying. We develop a light-weight energy management scheme called demand response with energy storage management   (\textsf{DR-ESM}), which does \emph{not} require any statistical knowledge of the system dynamics, and only requires  the user to \emph{solve a simple convex optimization problems with $6$ variables and $6$ linear constraints} for decision making. We also explicitly compute the required energy storage size, and prove that \textsf{DR-ESM} is able to achieve near-optimal performance with the chosen energy storage capacity.

This paper is organized as follows. In Section \ref{section:model}, we state our system model. In Section \ref{section:algorithm-esm}, we first develop the energy storage management algorithm (\textsf{ESM}) for optimal load serving and analyze its performance. We then extend our results to include demand response and present the demand response with energy storage management  algorithm (\textsf{DR-ESM}) in Section \ref{section:algorithm-esm-dr}. We generalize our results to Markovian systems in Section \ref{section:markov}. 
Simulation results are presented in Section \ref{section:simulation}. We then conclude our paper in Section \ref{section:conclusion}.   


\section{System Model}\label{section:model}

We consider a system where a power consuming entity (called user in the following), e.g., a group of residential users, or a commercial building, is trying to meet his  demand by purchasing power from the grid, and is trying to utilize his renewable energy and  a finite-capacity energy storage device to minimize his cost. The system is depicted in Fig. \ref{fig:storagebs}. We assume that the system operates in slotted time, i.e., $t\in\{0, 1, ...\}$. 
\begin{figure}[cht]
\centering
\includegraphics[height=1.4in, width=3.2in]{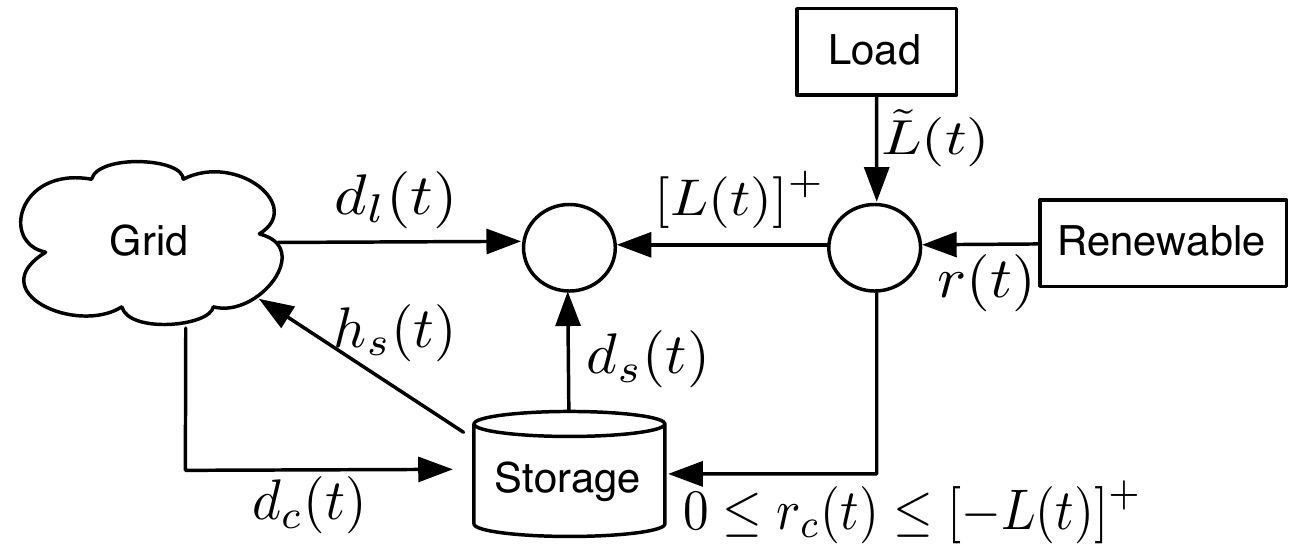}
\caption{\small{\textbf{User with energy storage}. 
$\tilde{L}(t)$ is the user load at time $t$, $r(t)$ is the available renewable energy. $L(t)\triangleq \tilde{L}(t)-r(t)$ is the \emph{residual demand}. $d_l(t)$ and $d_c(t)$ are the purchased power for serving the load and charging the storage, respectively. $h_s(t)$ is the power sold back to the grid  and $d_s(t)$ is the power drawn from the storage to serve the load. $r_c(t)$ is the amount from the excessive renewable energy to charge the energy storage. The objective of the user is to minimize his cost. 
}}\label{fig:storagebs}
\vspace{-.2in}
\end{figure}

\subsection{Disutility and load serving}
In every time slot, the user first decides how much load to consume. We denote this decision by $\tilde{L}(t)$ and assume $0\leq \tilde{L}(t)\leq L_{\textsf{max}}$. 
Depending on the system condition, however, certain power consumption level may incur some disutility to the user. For instance, suppose the user is operating a heater to warm his house. Then, if the outside temperature is low, operating the heater at a lower heating level, though being less costly, may lead to some discomfort for the user. 
 
To model such a system-dependent disutility and the fact that the system condition may be time-varying, we assume that there exists a  \emph{system state} $S(t)$ which captures the environmental aspect of the disutility function of the user, e.g., temperature.  Then,  the disutility of the user at time $t$ is determined by a general function $D(\tilde{L}(t), S(t))$. 
We assume that $S(t)$ takes value from some finite set $\script{S}$, and that the function $D(\tilde{L}(t), S(t))$ is convex in $\tilde{L}(t)$ for every system state $S(t)\in\script{S}$ and is known to the user. For instance, one example of $D(\tilde{L}(t), S(t))$ can be: 
\begin{eqnarray}
D(\tilde{L}(t), S(t)) = \beta_{S(t)} \big(L^{S(t)}_{\textsf{target}}-\tilde{L}(t) \big)^2. \label{eq:dis-utility-ex}
\end{eqnarray}
Here $L^{S(t)}_{\textsf{target}}$ denotes the target consumption level and $\beta_{S(t)}$ measures how the disutility increases as the power level deviates from the target level. 

We denote $r(t)$ the available renewable energy at time $t$, and define $L(t)\triangleq \tilde{L}(t)-r(t)$ to be the \emph{residual demand}. If $L(t)>0$ , then it is treated as a  normal load. Otherwise $-L(t)$ will be the excessive renewable energy, and can be used to charge the storage. 
\footnote{Note here we assume implicitly that we always use the renewable energy to serve the load whenever possible, and we also do not consider directly selling the renewable energy to the grid. Our results can easily be extended to the case when such direct selling is allowed.}  

%

The user can serve the residual demand $L(t)$ with two power  sources: power $d_l(t)$ purchased from the power grid at a unit-power price $p(t)$, and power $d_s(t)$ drawn from the energy storage at a zero price.  \footnote{Here we implicitly assume that the cost for operating the storage is zero. Our results can be extended to the case when there is cost associated with utilizing the storage. }
Similarly, the user can charge the energy storage with two power sources: power $d_c(t)$ purchased from the grid at price $p(t)$, and power $r_c(t)$ obtained from the excessive renewable energy.  
%
If the user stores enough energy, he can also sell some power back to the grid at a unit  price $q(t)$. We denote the amount sold back by  $h_s(t)$. 

In practice, there will be physical constraints on each of the power components. We model them as follows: 
\begin{itemize}
\item There is a capacity limit on how much power the user can draw from  the grid  at any time. We denote it by $c_{\textsf{grid}}$. Hence, in every time slot, we have: 
\begin{eqnarray}
d_l(t), d_c(t)\geq0, \,\, d_l(t)+d_c(t)\leq c_{\textsf{grid}}. \label{eq:grid-cond}
\end{eqnarray}
\item There is a maximum charging rate of the energy storage, denoted by  $c_{\textsf{char}}$, so that: 
\begin{eqnarray}
r_c(t)\geq0, \,\,d_c(t)+r_c(t)\leq c_{\textsf{char}}. \label{eq:charging-cond}
\end{eqnarray}
\item There is a maximum discharging rate of the energy storage, denoted by $c_{\textsf{dis}}$, so that: 
\begin{eqnarray}
h_s(t), d_s(t) \geq 0, \,\, h_s(t) + d_s(t)\leq c_{\textsf{dis}}. \label{eq:discharging-cond}
\end{eqnarray}
\end{itemize}
%
Besides these physical constraints, we note that at any time, the user's action must also ensure the following feasibility conditions: 
\begin{eqnarray}
d_l(t) + d_s(t) = [L(t)]^+, \,\, 0\leq r_c(t)\leq [-L(t)]^+. \label{eq:action-feasible}
\end{eqnarray}
That is, at every time slot, the load must be balanced using the power from the grid and the energy storage, and the amount of renewable energy used to charge the storage must be no more than the excessive renewable energy. 

We assume that  the renewable energy is bounded, i.e., $0\leq r(t)\leq r_{\textsf{max}}$. Then,  the residual load is also bounded, i.e.,  $-r_{\textsf{max}}\leq L(t)\leq L_{\textsf{max}}$. We also assume that the buying and selling prices are bounded for all time, i.e., $0\leq p(t) \leq p_{\textsf{max}}, 0\leq q(t)\leq q_{\textsf{max}}$. 

\subsection{Energy storage dynamics}
Under the charging/discharging actions, we assume that the energy storage level evolves according to the following dynamics: 
\begin{eqnarray}
E(t+1) = E(t) - \eta_e (d_s(t)+h_s(t)) + \eta_i (d_c(t)+r_c(t)). \label{eq:storage-dyn}
\end{eqnarray}
Here $\eta_e\geq1$ and $\eta_i\leq1$ are the coefficients associated with discharging and charging power from the storage. 
Note that by using (\ref{eq:storage-dyn}), we first assume that the energy storage device has an infinite capacity. Later we will show that under our algorithms, the energy queue level remains bounded \emph{deterministically} for all time and the needed storage size can be computed explicitly. Hence we only need a finite energy capacity to implement the algorithms.

Note that for the discharging decisions to be feasible,   the following energy-availability (\textsf{EA}) constraint must be met at all time:
\begin{eqnarray}
(\textsf{EA}):\,\, E(t) \geq \eta_e (d_s(t)+h_s(t)). \label{eq:storage-underflow} 
\end{eqnarray}
That is, there must be more stored energy than what is consumed. 

\vspace{-.1in}
\subsection{Objective}
Assuming that the (\textsf{EA}) condition is met, the user's instantaneous cost is given by its disutility plus the cost for purchasing power from the grid minus its gain from selling power, i.e., \footnote{Note that our results can also be extended to the case when $p(t)$ is a function of $d_l(t) +d_c(t)$. }
\begin{eqnarray}
f(t) \triangleq   D(\tilde{L}(t), S(t))+p(t) [d_l(t) +d_c(t)]  - q(t) h_s(t). \label{eq:cost-utility}
\end{eqnarray}
The user's objective is to find a control policy for determining the load consuming, purchasing/selling, and charging/discharging actions, so as to minimize his long term time average cost, defined:
\begin{eqnarray}
f_{\textsf{av}} \triangleq \limsup_{t\rightarrow\infty} \frac{1}{t}\sum_{\tau=0}^{t-1} \expect{f(\tau)}. \label{eq:avg-cost}
\end{eqnarray}
We call an energy management scheme that ensures (\textsf{EA}) for all time a \emph{feasible} scheme, and use $f_{\textsf{av}}^*$ to denote the infimum average cost that any feasible policy can achieve with a finite capacity energy storage. 

For ease of presentation, below we first assume that the quadruple $(p(t), q(t), r(t), S(t))$ is  i.i.d. over each time slot. However, we allow the components to be arbitrarily correlated, which is typically the case in practice. Later we will relax the i.i.d. assumption and  generalize the results to the case when $(p(t), q(t), r(t), S(t))$  is Markovian.  
Finally, we assume the following on the system parameters: 
\begin{eqnarray}
 \eta_i c_{\textsf{grid}}\geq \eta_eL_{\textsf{max}}. \label{eq:capacity-assumption}
\end{eqnarray}
Since $\eta_e$ is the efficiency for power discharging, we see that (\ref{eq:capacity-assumption}) can be viewed as saying that,  if the storage can be charged with the full power rate $c_{\textsf{grid}}$, then it can be used to support a full load for one time slot. Such a condition can typically be satisfied in practice, since $c_{\textsf{grid}}$ is usually large compared to $L_{\textsf{max}}$. 

\vspace{-.1in}
\subsection{Discussion of the model}
Our model is similar to those used in \cite{elgamal-storage11}, \cite{rahulneely-storage},  and  \cite{huangneely-energy-mobihoc}. 
However, in \cite{elgamal-storage11}, the problem is solved using dynamic programming (DP) under various approximations of the charging/discharging capacity constraints and without demand response. In  \cite{rahulneely-storage} and \cite{huangneely-energy-mobihoc}, the storage is only used to support the user's load but not as an economic tool for the user. Thus, they do not consider the possibility of selling the stored power back to the grid as well as demand response. 
%

In the following, in order to best demonstrate our solution approach, we first consider the case when the user does not perform demand response, i.e., he only tries to serve the given load $\tilde{L}(t)$ every time slot.  
This is an important scenario which models cases when the power consumption level cannot be changed. The results derived for this load-serving case will  demonstrate our  main solution idea, and will be extended to incorporate demand response in Section \ref{section:algorithm-esm-dr}. 


\section{Algorithm Design for load serving}\label{section:algorithm-esm}
In this section, we first consider the case when $\tilde{L}(t)$ is a given stochastic process that satisfies $0\leq \tilde{L}(t)\leq L_{\textsf{max}}$, and that $(p(t), q(t), r(t), S(t), \tilde{L}(t))$ is an i.i.d. process. This corresponds to the case when the load is non-deferrable and has to be balanced in every time slot. In this case, the instantaneous cost becomes:
\begin{eqnarray}
f(t) \triangleq p(t) [d_l(t) +d_c(t)]  - q(t) h_s(t). \label{eq:cost-utility-load}
\end{eqnarray}

We will solve the following two important problems at the same time: (i) How should we size the capacity of the storage device according to the system parameters and performance requirements? (ii) Given the energy storage capacity, how can we optimally control the storage device with low-complexity algorithms that adapt quickly to the system dynamics and provide performance guarantees? 

These two problems are challenging. 
Indeed, the main difficulty in resolving them is the (\textsf{EA}) constraint which requires that no energy underflow happens for all time. This constraint 
couples all the actions across time. Hence, most prior works tackle it using DP, which typically requires substantial information of the system dynamics, and is  computationally expensive. 
We instead construct our algorithm based on the Lyapunov optimization approach used in \cite{rahulneely-storage} and \cite{huangneely-energy-mobihoc}. The main idea of our approach is to ``temporarily ignore'' the (\textsf{EA}) constraint and solve  the problem,  but then show that the constraint is automatically guaranteed by our algorithm. 
We will see that, using this approach, we are able to derive very light-weight algorithms that can easily be implemented in practice and guarantee near-optimal performance. 

\vspace{-.1in}
\subsection{The energy storage management algorithm (\textsf{ESM})}
To start, we first define two control parameters $\theta>0$ and $\epsilon>0$ (values to be specified later), where $\epsilon$ is used to control the distance between the performance of our algorithm and the optimum, and $\theta$ is used to determine the optimal energy consumption. 
We then define a Lyappunov function $G(t)=\frac{1}{2}(E(t)-\theta)^2$, and define the following one-step Lyapunov drift:
\begin{eqnarray}
\Delta(t) = \expect{G(t+1)-G(t)\left.|\right. E(t)}. 
\end{eqnarray}
The following lemma first obtains a basic property of the Lyaunov drift. 
\begin{lemma}\label{lemma:basic}
The Lyapunov drift satisfies the following inequality for all time: 
\begin{eqnarray}
\hspace{-.2in}&&\Delta(t) \leq B - (E(t) -\theta ) \expect{  \eta_e d_s(t) + \eta_eh_s(t)\label{eq:basic-drift}\\
\hspace{-.2in}&& \qquad\qquad\qquad\qquad\qquad\quad - \eta_i d_c(t) -\eta_ir_c(t) \left.|\right. E(t)}, \nonumber
\end{eqnarray}
where $B\triangleq\frac{1}{2}[\eta_e^2c_{\textsf{dis}}^2 +  \eta_i^2c_{\textsf{char}}^2]$. $\Diamond$
\end{lemma}
\begin{proof}
See Appendix A. 
\end{proof}
To construct our algorithm, we define $V=1/\epsilon$, and add to both sides of (\ref{eq:basic-drift}) the term $V\expect{ f(t) \left.|\right. E(t)}$ to get:
\begin{eqnarray}
\hspace{-.2in}&&\Delta(t) + V\expect{ f(t) \left.|\right. E(t)} \label{eq:esm-drift-foo1}\\
\hspace{-.2in}&&\leq B + V\expect{p(t) [d_l(t) +d_c(t)]  - q(t) h_s(t) \left.|\right. E(t)}\nonumber\\
\hspace{-.2in}&&\qquad\, -\, (E(t) -\theta ) \expect{  \eta_e (d_s(t) + h_s(t)) \nonumber\\
\hspace{-.2in}&& \qquad\qquad\qquad\qquad\qquad\quad  - \eta_i d_c(t) -\eta_ir_c(t) \left.|\right. E(t)}\nonumber\\
\hspace{-.2in}&&= B - \expect{d_s(t)\eta_e(E(t)-\theta)\left.|\right. E(t)} \nonumber\\
\hspace{-.2in}&& \qquad-\,\, \expect{ h_s(t) [\eta_e(E(t)-\theta) + Vq(t) ]  \left.|\right. E(t)} \nonumber\\
\hspace{-.2in}&&   \qquad +\,\,  \expect{  d_l(t)Vp(t) \left.|\right. E(t)} \nonumber\\
\hspace{-.2in}&& \qquad + \,\, \expect{d_c(t) [ Vp(t) +\eta_i(E(t)-\theta)]     \left.|\right. E(t)}\nonumber\\
\hspace{-.2in}&& \qquad + \,\,\expect{r_c(t)\eta_i(E(t)-\theta)   \left.|\right. E(t) }. \nonumber
\end{eqnarray}
We then replace $d_l(t)=[L(t)]^+-d_s(t)$ in the above to get:
\begin{eqnarray}
\hspace{-.3in}&&\Delta(t) + V\expect{ f(t) \left.|\right. E(t)} \label{eq:esm-drift-foo-ref}\\
\hspace{-.3in}&&\leq  B + \expect{V[L(t)]^+p(t)\left.|\right.E(t)}  \nonumber\\
\hspace{-.3in}&&\qquad -\,\, \expect{ h_s(t) [\eta_e(E(t)-\theta) + Vq(t) ]  \left.|\right. E(t)} \nonumber\\
\hspace{-.3in}&& \qquad -\,\, \expect{d_s(t) [\eta_e(E(t)-\theta) +Vp(t)] \left.|\right. E(t)} \nonumber\\
\hspace{-.3in}&&\qquad + \,\, \expect{d_c(t) [ Vp(t) +\eta_i(E(t)-\theta)]     \left.|\right. E(t)}\nonumber\\
\hspace{-.2in}&& \qquad + \,\,\expect{r_c(t)\eta_i(E(t)-\theta)   \left.|\right. E(t) }. \nonumber
\end{eqnarray}
Our power management algorithm is then constructed using the ``min-drift'' principle of the Lyapunov optimization technique \cite{neelynowbook}:  at every time slot, choose a set of feasible charging/discharging and buying/selling actions to  minimize  the right-hand-side (RHS) of (\ref{eq:esm-drift-foo-ref}). Doing so, we obtain the following algorithm:

\underline{\textsf{Energy Storage Management (ESM)}:} Use an energy storage of capacity $\theta+\eta_ic_{\textsf{char}}$ ($\theta$ will be specified later). At every time $t$, do: 
\begin{enumerate}
\item Observe the energy level $E(t)$, the residual load $L(t)$, and the power prices $p(t)$ and $q(t)$. \footnote{Note here observing the residual load $L(t)$ includes observing the renewable energy $r(t)$. This can be done  accurately since we are observing its instant value, which is equivalent to performing short time forecast.}  
 Define the following weights: 
\begin{eqnarray}
\hspace{-.6in}&&W_h(t) = \eta_e(E(t)-\theta) + \frac{q(t)}{\epsilon}, W_s(t) = \eta_e(E(t)-\theta) +\frac{p(t)}{\epsilon},\nonumber\\
\hspace{-.6in}&&W_c(t) = \eta_i(E(t)-\theta) + \frac{p(t)}{\epsilon}, W_r(t) = \eta_i(E(t)-\theta). \label{eq:weight-def}
\end{eqnarray}

\item Chooses $h_s(t)$, $d_l(t)$, $d_s(t)$, $d_c(t)$, and $r_c(t)$ to solve the following optimization problem subject to (\ref{eq:grid-cond}), (\ref{eq:charging-cond}), (\ref{eq:discharging-cond}), (\ref{eq:action-feasible}). Specifically, we solve: 
\begin{eqnarray}
\hspace{-.6in}&&\max: \,\,\, h_s(t) W_h(t) + d_s(t) W_s(t) \label{eq:esm-1}\\
\hspace{-.6in}&&\qquad\qquad\qquad\qquad\,\,  - d_c(t) W_c(t) - r_c(t)W_r(t)\nonumber\\
\hspace{-.6in}&&\quad\,\,\text{s.t.}\,\,\,\,  d_l(t) + d_s(t) = [L(t)]^+, \label{eq:esm-cond-1}\\
\hspace{-.6in}&&\qquad\quad\,\,  d_l(t)+d_c(t)\leq c_{\textsf{grid}}, \label{eq:esm-cond-2}\\
\hspace{-.6in}&&\qquad\quad\,\,  d_c(t)+r_c(t)\leq c_{\textsf{char}}, \label{eq:esm-cond-3}\\ 
\hspace{-.6in}&&\qquad\quad\,\,  h_s(t) + d_s(t)\leq c_{\textsf{dis}}, \label{eq:esm-cond-4}\\
\hspace{-.6in}&&\qquad\quad\,\,  r_c(t)\leq [-L(t)]^+, \label{eq:esm-cond-5}\\
\hspace{-.6in}&& \qquad\quad\,\,  d_l(t), d_c(t), h_s(t), d_s(t), r_c(t)\geq0. \nonumber
\end{eqnarray}
Then, perform the chosen actions. $\Diamond$
\end{enumerate}
Note that the complexity of $\textsf{ESM}$ is very low: at every time slot,  \emph{the user only has to solve a linear program with $5$ variables and $5$ linear constraints (\ref{eq:esm-cond-1})-(\ref{eq:esm-cond-5})}. This is a very simple task and can easily  be done. It also \emph{does not} require any statistical knowledge of the residual load $L(t)$, the renewable energy $r(t)$, and the power prices $p(t)$ and $q(t)$. Hence, it can easily be implemented in practice. We also note that  \textsf{ESM} does not explicitly take into account the (\textsf{EA})  constraint. However, we will show later that (\textsf{EA}) is automatically ensured by the \textsf{ESM} algorithm. 

\vspace{-.1in}
\subsection{Performance analysis of \textsf{ESM}}
The following theorem summarizes the performance of \textsf{ESM}. In the theorem, the parameter $\theta$ is defined: 
\begin{eqnarray}
\theta\triangleq\frac{\max[p_{\textsf{max}}, q_{\textsf{max}}]}{\epsilon  \eta_i} + \eta_e\min[L_{\textsf{max}}, c_{\textsf{dis}}]. \label{eq:theta-def}
\end{eqnarray}
\begin{theorem}\label{theorem:per-esm}
Suppose $\theta$ is chosen according to (\ref{eq:theta-def}), and $0\leq E(0)\leq \theta+\eta_ic_{\textsf{char}}$. 
Then, under \textsf{ESM}, we have: 
\begin{eqnarray}
\hspace{-.2in}&&0\leq E(t)\leq \theta +\eta_ic_{\textsf{char}}, \,\,\forall\,t, \label{eq:energy-bound}\\
\hspace{-.2in}&& \quad\,\,\, f_{\textsf{av}}^{\textsf{ESM}} \leq f_{\textsf{av}}^*+B\epsilon.\label{eq:cost-bound}
\end{eqnarray}
Here $f_{\textsf{av}}^{\textsf{ESM}}$ is the average cost achieved by \textsf{ESM}, $f_{\textsf{av}}^*$ is the optimal time average cost, and $B=\frac{1}{2}[\eta_e^2c_{\textsf{dis}}^2 +  \eta_i^2c_{\textsf{char}}^2]$. $\Diamond$
\end{theorem}

Note that (\ref{eq:energy-bound}) is very important. It shows that the energy level under \textsf{ESM} will never be negative and is \emph{deterministically} upper bounded, and provides the explicitly bound. 
This shows that the (\textsf{EA}) condition is ensured under \textsf{ESM}, and 
allows us to conveniently size our storage capacity and  implement the algorithm with a finite capacity energy storage. We also emphasize that (\ref{eq:energy-bound}) is indeed a \emph{sample path} result. Thus it holds under \emph{arbitrary} $(\tilde{L}(t), p(t), q(t), r(t), S(t))$ processes. This shows that our algorithm is indeed applicable under more general system dynamics. 

Below we present the proof of (\ref{eq:energy-bound}). The proof of (\ref{eq:cost-bound}) will be given in Appendix B. 
\begin{proof} (Theorem \ref{theorem:per-esm}) 
Let $d_l^*(t), d^*_c(t), d^*_s(t), h_s^*(t), r^*_c(t)$ be an optimal solution of (\ref{eq:esm-1}). 

We first prove the upper bound of (\ref{eq:energy-bound}) using induction. 
First we see that it holds for time $t=0$. 
Now assume the upper bound holds at time $t$. 
\begin{enumerate}
\item Suppose $E(t)\leq\theta$. Then, since $d_c(t)+r_c(t)\leq c_{\textsf{char}}$, we see that $E(t+1)\leq \theta+\eta_ic_{\textsf{char}}$. 
\item Now suppose $E(t)>\theta$. Then we have $W_c(t), W_r(t)>0$. Thus, we must have $d^*_c(t)=0$ and  $r^*_c(t)=0$. That is, once $E(t)>\theta$, the energy level will not further increase. Hence, $E(t+1)\leq E(t)\leq \theta+\eta_ic_{\textsf{char}}$. 
\end{enumerate}

Now we prove the lower bound using induction. Assume the lower bound holds at time $t$. 
\begin{enumerate}
\item Suppose $E(t)\geq \eta_e\min[L_{\textsf{max}}, c_{\textsf{dis}}]$. Since  the maximum amount that can be discharged from the storage is $ \eta_e\min[L_{\textsf{max}}, c_{\textsf{dis}}]$, we see that $E(t+1)\geq0$.

\item  Suppose $E(t)< \eta_e\min[L_{\textsf{max}}, c_{\textsf{dis}}]$. In this case, we see from (\ref{eq:theta-def}) that:  
\begin{eqnarray}
\eta_e(E(t)-\theta), \eta_i(E(t)-\theta)<-  \frac{\max[p_{\textsf{max}}, q_{\textsf{max}}]}{\epsilon}.\label{eq:theta-foo}
\end{eqnarray}
Using (\ref{eq:weight-def}), we see that: 
\begin{eqnarray}
W_h(t), W_s(t), W_c(t), W_r(t)<0. 
\end{eqnarray}
In this case, we first see from (\ref{eq:esm-cond-4}) that $h^*_s(t)=0$, i.e., no energy is sold back to the grid. 
%

Then, we see that either (i) $d_c^*(t)+r^*_c(t)=c_{\textsf{char}}$ or (ii) $d^*_l(t)+d_c^*(t)=c_{\textsf{grid}}$. 
This is so because if none of them happens,  then we can increase $d_c^*(t)$ to further increase the value of the objective function of (\ref{eq:esm-1}), which contradicts the fact that $d^*_c(t)$ is the optimal solution. 

Suppose now $d_c^*(t)+r_c^*(t)=c_{\textsf{char}}$.  Then, from the objective function (\ref{eq:esm-1}) and the constraints (\ref{eq:esm-cond-1}) and (\ref{eq:esm-cond-2}), we see that $d^*_l(t)$ should be as large as possible, so that we can use a minimum $d_s^*(t)$ to maximize (\ref{eq:esm-1}). Hence, we have: 
\begin{eqnarray}
d^*_l(t) = \min[c_{\textsf{grid}}- d^*_c(t), [L(t)]^+]. 
\end{eqnarray}
This implies that either $d^*_l(t)+d_c^*(t)=c_{\textsf{grid}}$ or $d^*_l(t)=[L(t)]^+$. If $d^*_l(t)=[L(t)]^+$, then $d_s^*(t)=0$. Since we also have $h^*_s(t)=0$ and $d_c^*(t), r_c^*(t)\geq0$, we see that: 
\begin{eqnarray}
\hspace{-.4in}&&E(t+1) = E(t) - \eta_e (d^*_s(t)+h^*_s(t)) + \eta_i (d^*_c(t)+r^*_c(t))\nonumber\\
\hspace{-.4in}&&\qquad\qquad \geq E(t). \nonumber
\end{eqnarray}
Now suppose instead we have $d^*_l(t)+d_c^*(t)=c_{\textsf{grid}}$. Using the facts that  $d_l^*(t)+d^*_s(t)=[L(t)]^+$ and $[L(t)]^+\leq L_{\textsf{max}}$, we have: 
\begin{eqnarray}
\hspace{-.6in}&&E(t+1) = E(t) - \eta_e (d^*_s(t)+h^*_s(t)) + \eta_i (d^*_c(t)+r^*_c(t))\nonumber\\ 
\hspace{-.6in}&&\qquad\qquad = E(t) - \eta_e ([L(t)]^+-d^*_l(t)) + \eta_i (d^*_c(t)+r^*_c(t))\nonumber\\ 
\hspace{-.6in}&& \qquad\qquad  \geq E(t) -\eta_eL_{\textsf{max}} + \eta_i(d^*_l(t)+d^*_c(t))\nonumber\\
\hspace{-.6in}&&\qquad\qquad = E(t) -\eta_eL_{\textsf{max}} +  \eta_i c_{\textsf{grid}}\nonumber\\
\hspace{-.6in}&& \qquad\qquad \geq E(t), \label{eq:lower-bound-foo}
\end{eqnarray}
where the last step follows from (\ref{eq:capacity-assumption}). This thus means that whenever $E(t)< \eta_e\min[L_{\textsf{max}}, c_{\textsf{dis}}]$, the energy level will not decrease. Hence $E(t+1)\geq E(t)\geq0$. 
%
\end{enumerate}
This completes the proof of  (\ref{eq:energy-bound}). 
\end{proof}

From the above proof, we notice that whenever $E(t)>\theta$, \textsf{ESM} will set $r_c(t)=0$. This is an important feature in that it guarantees the boundedness of the stored energy level. Hence, \textsf{ESM} can be implemented with finite energy storage devices. 
The reason why \textsf{ESM} can afford not to charge the energy storage with the free renewable energy is because under the chosen capacity and \textsf{ESM}, the loss of optimality will not be more than $O(\epsilon)$ even if we occasionally waste a small fraction of the renewable energy. In practice, we can modify \textsf{ESM} to 
always store excessive renewable energy, and control the system with a ``virtual'' control sequence that tracks the energy level under the original \textsf{ESM} algorithm.

\vspace{-.1in}
\section{Energy management with demand response}\label{section:algorithm-esm-dr}
In this section, we extend our results to the case when the user also performs demand response. Specifically, in every time slot,  the user also chooses his power consumption level $0\leq\tilde{L}(t)\leq L_{\textsf{max}}$, and is willing to change his consumption according to the system condition. 
Recall that in this case, the user's instant cost becomes: 
\begin{eqnarray}
f(t) \triangleq   D(\tilde{L}(t), S(t))+p(t) [d_l(t) +d_c(t)]  - q(t) h_s(t), \label{eq:cost-utility-dr}
\end{eqnarray}
where $D(\tilde{L}(t), S(t))$ is the disutility experienced by the user for consuming a power level $\tilde{L}(t)$.

\vspace{-.1in}
\subsection{\textsf{ESM} with demand response}
Here we construct our algorithm to incorporate demand response. 
Using a similar argument as before, we have: 
\begin{eqnarray*}
\hspace{-.3in}&&\Delta(t) + V\expect{ f(t) \left.|\right. E(t)} \\
\hspace{-.3in}&&\leq B -\expect{h_s(t) [\eta_e(E(t)-\theta) + Vq(t)]  \left.|\right. E(t)} \\
\hspace{-.3in}&&  \qquad+ \,\,\expect{ d_l(t) Vp(t)  \left.|\right. E(t)}  \\
\hspace{-.3in}&&  \qquad- \,\,\expect{d_s(t)\eta_e(E(t)-\theta)  \left.|\right. E(t)} \\
\hspace{-.3in}&&\qquad + \,\,\expect{d_c(t) [ \eta_i(E(t)-\theta) + Vp(t)]   \left.|\right. E(t)} \\
\hspace{-.3in}&& \qquad + \,\,\expect{r_c(t)\eta_i(E(t)-\theta)   \left.|\right. E(t) }\\
\hspace{-.3in}&& \qquad + \,\,\expect{VD(\tilde{L}(t), S(t)) \left.|\right. E(t)}.  
\end{eqnarray*}
Replacing $d_s(t)=[L(t)]^+ - d_l(t)$, and rearranging the terms, we have:
\begin{eqnarray*}
\hspace{-.3in}&&\Delta(t) + V\expect{ f(t) \left.|\right. E(t)} \\
\hspace{-.3in}&&\leq B -\expect{h_s(t) [\eta_e(E(t)-\theta) + Vq(t)]  \left.|\right. E(t)} \\
\hspace{-.3in}&&  \qquad+ \,\,\expect{ d_l(t) [Vp(t)  +\eta_e(E(t)-\theta)  ] \left.|\right. E(t)}  \\
\hspace{-.3in}&&\qquad + \,\,\expect{d_c(t) [ \eta_i(E(t)-\theta) + Vp(t)]   \left.|\right. E(t)} \\
\hspace{-.3in}&& \qquad + \,\,\expect{r_c(t)\eta_i(E(t)-\theta)   \left.|\right. E(t) }\\
\hspace{-.3in}&& \quad + \,\,\expect{VD(\tilde{L}(t), S(t)) - \eta_e(E(t)-\theta)\big[\tilde{L}(t)-r(t)\big]^+  \left.|\right. E(t)}. 
\end{eqnarray*}
We now similarly construct our algorithm by choosing actions to minimize the RHS of the drift inequality at every time slot. Doing so, we obtain the following algorithm.  

\underline{\textsf{Demand Response with ESM  (DR-ESM)}:} Use an energy storage of capacity $\theta+\eta_ic_{\textsf{char}}$. At every time $t$, do: 
\begin{enumerate}
\item Observe the energy level $E(t)$, the renewable energy $r(t)$ and the prices $p(t)$ and $q(t)$. Define the following weights: 
\begin{eqnarray}
\hspace{-.6in}&&W_h(t) = \eta_e(E(t)-\theta) + \frac{q(t)}{\epsilon}, \,W_l(t) = \eta_e(E(t)-\theta) +\frac{p(t)}{\epsilon},  \nonumber\\
\hspace{-.6in}&&W_c(t) = \eta_i(E(t)-\theta) + \frac{p(t)}{\epsilon}, \,W_r(t) =  \eta_i(E(t)-\theta), \label{eq:weight-def-dr}\\
\hspace{-.6in}&&W_D(t) = \eta_e(E(t)-\theta).  \nonumber
\end{eqnarray}

\item Chooses $\tilde{L}(t)$, $h_s(t)$, $d_l(t)$, $d_s(t)$, $d_c(t)$ and $r_c(t)$ to solve:
\begin{eqnarray}
\hspace{-.6in}&&\min:  VD(\tilde{L}(t), S(t)) -W_D(t)[\tilde{L}(t) - r(t)]^+  \label{eq:esm-dr}\\
\hspace{-.6in}&&\qquad\,\,\,-h_s(t) W_h(t) + d_l(t) W_l(t)+ d_c(t) W_c(t)+r_c(t)W_r(t)\nonumber\\
\hspace{-.6in}&&\,\,\text{s.t.}\quad  (\ref{eq:grid-cond}), (\ref{eq:charging-cond}), (\ref{eq:discharging-cond}), (\ref{eq:action-feasible}),  0\leq\tilde{L}(t)\leq L_{\textsf{max}}. \nonumber
\end{eqnarray}
Then, perform the chosen actions. 
$\Diamond$
\end{enumerate}
Notice that even with demand response, the \textsf{DR-ESM} algorithm still does not require any statistical knowledge of the system variables, and 
only requires the user  to solve a simple convex optimization program \emph{with $6$ variables and $6$ constraints}, which can be solved very efficiently. 
%

The performance of \textsf{DR-ESM} is summarized in the following theorem. 
\begin{theorem}\label{theorem:per-esm-dr}
Suppose $\theta$ is chosen according to (\ref{eq:theta-def}), and $0\leq E(0)\leq \theta+\eta_ic_{\textsf{char}}$. Then, under \textsf{DR-ESM}, we have: 
\begin{eqnarray}
\hspace{-.2in}&&0\leq E(t)\leq \theta + \eta_ic_{\textsf{char}},\,\,\forall\,t, \label{eq:energy-bound-dr}\\
\hspace{-.2in}&& \quad\,\,\, f_{\textsf{av}}^{\textsf{E-DR}} \leq f_{\textsf{av}}^*+B\epsilon.\label{eq:cost-bound-dr}
\end{eqnarray}
Here $f_{\textsf{av}}^{\textsf{E-DR}}$ is the time average cost achieved by \textsf{DR-ESM}, and $B=\frac{1}{2}[\eta_e^2c_{\textsf{dis}}^2 +  \eta_i^2c_{\textsf{char}}^2]$. $\Diamond$
\end{theorem}
\begin{proof}
See Appendix B. 
\end{proof}

\section{\textsf{ESM} with Markovian system dynamics}\label{section:markov}
Here we extend the our results to the case when the system parameters, i.e., the quadruple  $(p(t), q(t), r(t), S(t))$ evolves according to a finite state irreducible and aperiodic  Markov chain. In this case, the following theorem shows that \textsf{DR-ESM} (and similarly \textsf{ESM}) still achieves similar performance as in the i.i.d. case. 
\begin{theorem}
Suppose  $(p(t), q(t), r(t), S(t))$ evolves according to a finite state irreducible and aperiodic  Markov chain. Also suppose $\theta$ is chosen according to (\ref{eq:theta-def}) and $0\leq E(0)\leq \theta+\eta_ic_{\textsf{char}}$. 
Then, under  \textsf{DR-ESM}, we have: 
\begin{eqnarray}
\hspace{-.2in}&&0\leq E(t)\leq \theta +\eta_ic_{\textsf{char}}, \,\,\forall\,t, \label{eq:energy-bound-markov}\\
\hspace{-.2in}&& \quad\,\,\, f_{\textsf{av}}^{\textsf{E-DR}} \leq f_{\textsf{av}}^*+O(\epsilon).\label{eq:cost-bound-markov}
\end{eqnarray}
Here $f_{\textsf{av}}^*$ is the optimal time average cost. $\Diamond$
\end{theorem}
\begin{proof}
The result (\ref{eq:energy-bound-markov}) follows directly from Theorem \ref{theorem:per-esm} and Theorem \ref{theorem:per-esm-dr} because it is a sample path result. (\ref{eq:cost-bound-markov}) can be proven using the variable multi-slot drift argument development in \cite{huangneely_qlamarkovian}. The details are omitted for brevity. 
\end{proof}

\section{Simulation}\label{section:simulation}
Here we present simulation results. For simplicity, we only simulate 
\textsf{DR-ESM}. 
We assume that the load consuming entity's consumption level is equivalent to  $10$ residential users, e.g., the entity is an apartment building, or a small commercial building. Every time slot is assumed to be one hour.  

We generate time-varying hourly buying prices by assuming that the daily average hourly price is $12\cent/$kWh (according to PG$\&$E's residential electricity report \cite{pge-web}) and that the distribution is uniform over the annual average day-ahead hourly price of PG$\&$E in the year $2010$ (data from the CAISO daily report on the Federal Energy Regulatory Commission (FERC) website \cite{caiso-web}). The values are shown in Fig. \ref{fig:price-renewable}. 
 We then assume that the selling price at every time is the same as the buying price. 
\begin{figure}[cht]
\centering
\vspace{-.1in}
\includegraphics[height=1.9in, width=3.3in]{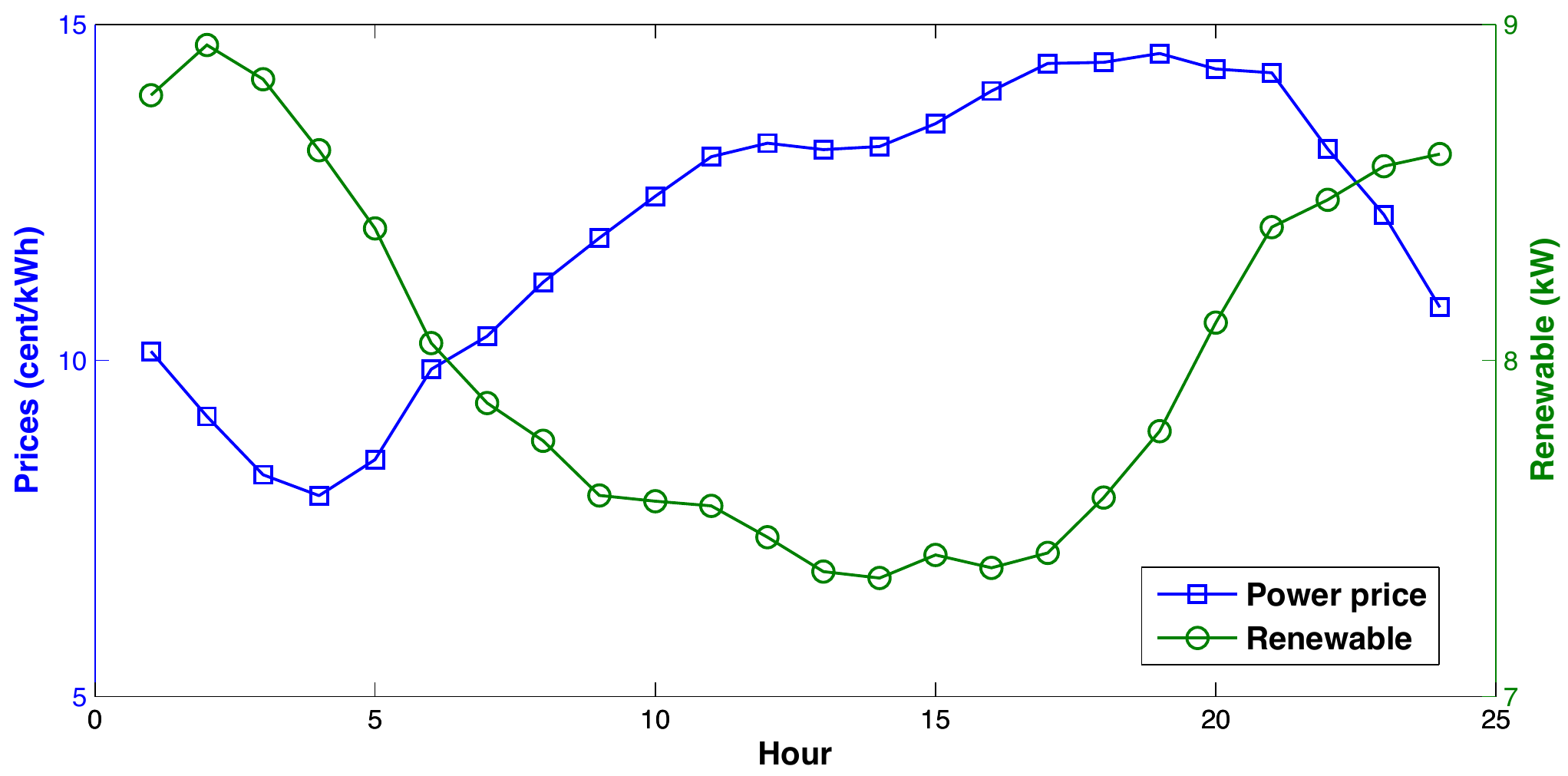}
\caption{The price distribution and the renewable energy pattern}\label{fig:price-renewable}
\end{figure}

We assume that the user has a wind turbine with a power capacity of $9$kW, i.e., $r_{\textsf{max}}=9$. \footnote{Since the entity roughly represents $10$ household users, this capacity can represent the total capacity of several home wind turbines.} 
We similarly assume that the wind power has a uniform distribution over the values shown in the  curve in Fig. \ref{fig:price-renewable}, which is obtained by averaging and normalizing the $2006$ annual recorded wind power of $100$ wind turbines located near $40.42N, 124.39W$ at  the west coast in California, which has an aggregate capacity of $300$MW (data from the National Renewable Energy Laboratory (NREL) website \cite{wind-web}), and has a mean value $8$kW. 
We assume that $\eta_e=1.25$ and $\eta_i=0.8$, corresponding to $80\%$ efficiency for both charging and discharging for the storage, and that $c_{\textsf{grid}}=20$kW, $c_{\textsf{char}}=c_{\textsf{dis}}=12$kW. 

We assume that $S(t)$ takes two values ``H=High'' and ``L=Low'' with equal probabilities, e.g., representing high or low temperature, and that the disutility function is of the form given in (\ref{eq:dis-utility-ex}) with $\beta_{H}=\beta_L=1$, i.e., 
\begin{eqnarray}
D(\tilde{L}(t), S(t)) =  \big( L^{S(t)}_{\textsf{target}}-\tilde{L}(t) \big)^2,  
\end{eqnarray}
with $L^{H}_{\textsf{target}}=12$kW, and $L^{L}_{\textsf{target}}=8$kW. We assume that $L_{\textsf{max}}=12$kW. \footnote{Since we assume each slot is one hour, expressing the consumption level in units of kW is the same as in kWh.} These numbers are chosen based on the overall  $2009$ California residential average power consumption, which is $18.6$kWh per day per user, or $0.8$kWh per hour per user \cite{saturationstudy-kema}.  To make the units consistent, we assume that the utility is also measured in units of cents. 

According to (\ref{eq:theta-def}), we use a storage of size $\theta+\eta_ic_{\textsf{char}}=18V+24.6$kWh. 
Below we simulate $V=\{2, 5, 10, 20, 50\}$. 
Our simulation is done with \texttt{CVX} \cite{cvx}, and each simulation is run for $10^4$ slots. 
For comparison purposes, we also simulate the case when there is no energy storage. In this case, we assume that the user at every time slot chooses the load according to the following \textsf{Greedy} scheme: 
\begin{eqnarray}
\tilde{L}(t)\in\arg\min D(\tilde{L}(t), S(t)) + p(t)(\tilde{L}(t) - r(t))^+. 
\end{eqnarray}
That is, the user is trying to minimize the instant cost every time slot. 

\begin{figure}[cht] 
\centering
\vspace{-.1in}
\includegraphics[height=1.9in, width=3.3in]{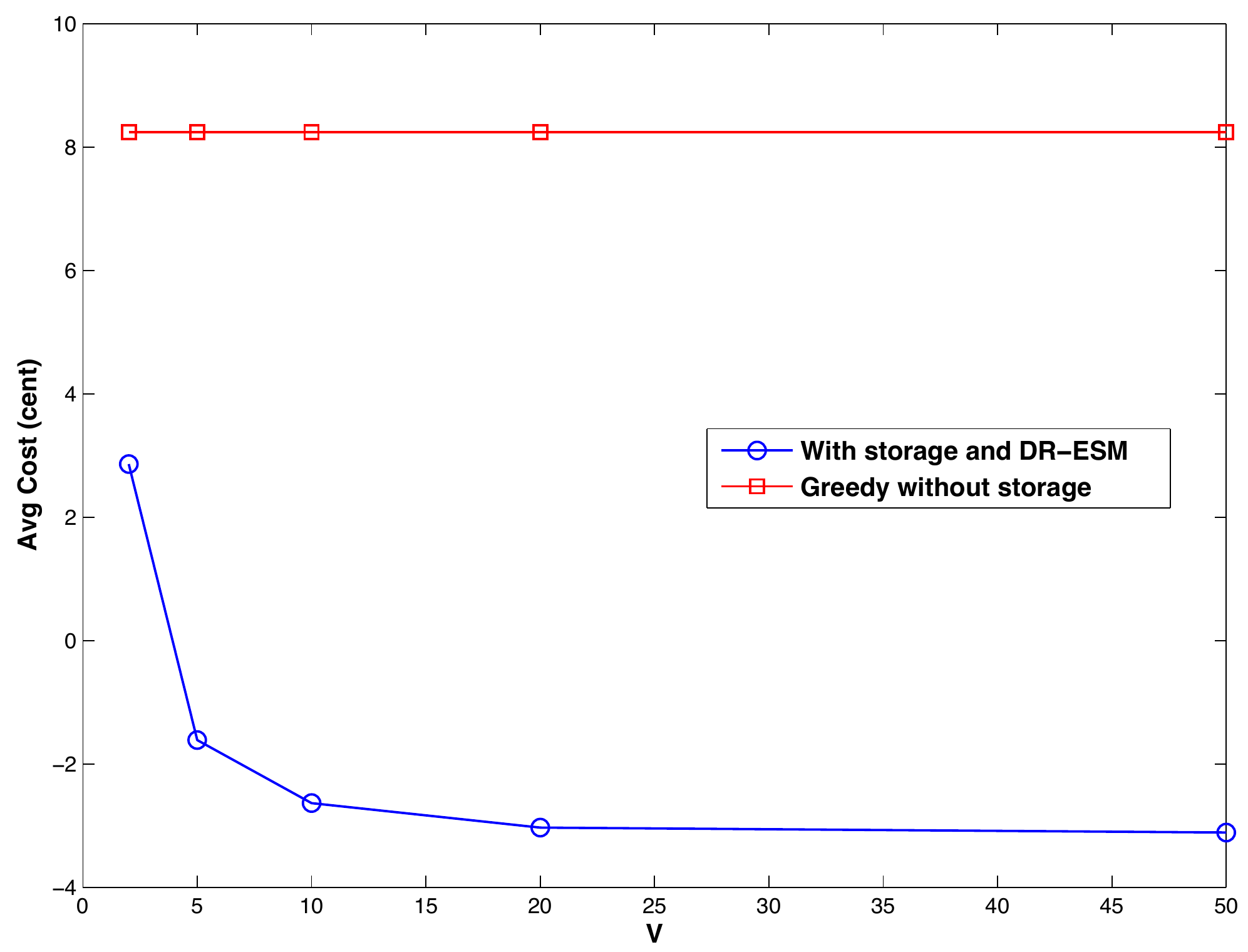}
\caption{Average cost under \textsf{DR-ESM} and without storage.}\label{fig:avg-cost}
\end{figure}

Fig. \ref{fig:avg-cost} shows the average costs under \textsf{DR-ESM} and \textsf{Greedy}. We see that  \textsf{DR-ESM} is able to reduce the average cost by $64\%-136\%$. For instance, when $V=5$, \textsf{DR-ESM} reduces the average cost from $8.24\cent$ down to $-1.61\cent$, which corresponds to a saving of $120\%$. 
%
The reason that the reduction can exceed $100\%$ is because with energy storage and \textsf{DR-ESM}, the user can actually make profit by carefully buying and selling power. 
We also note that such significant saving is achieved with a moderate storage capacity. For instance, when $V=5$, the provisioned storage capacity is $118.35$kWh. Considering the fact that the user represents roughly $10$ residential users, this only requires each individual user to have a battery of size $12$kWh.

Fig. \ref{fig:energy-level} also shows a sample path energy level process under \textsf{DR-ESM} with $V=5$ in time slots $[101, 400]$. We note that the energy level is always below the provisioned capacity and never drops below $0$. 
One interesting observation is that we can even implement \textsf{DR-ESM} with a storage of size $75$kWh.  This indicates that our algorithm can likely be implemented with much smaller storages. 
\begin{figure}[cht] 
\centering
\vspace{-.1in}
\includegraphics[height=1.9in, width=3.3in]{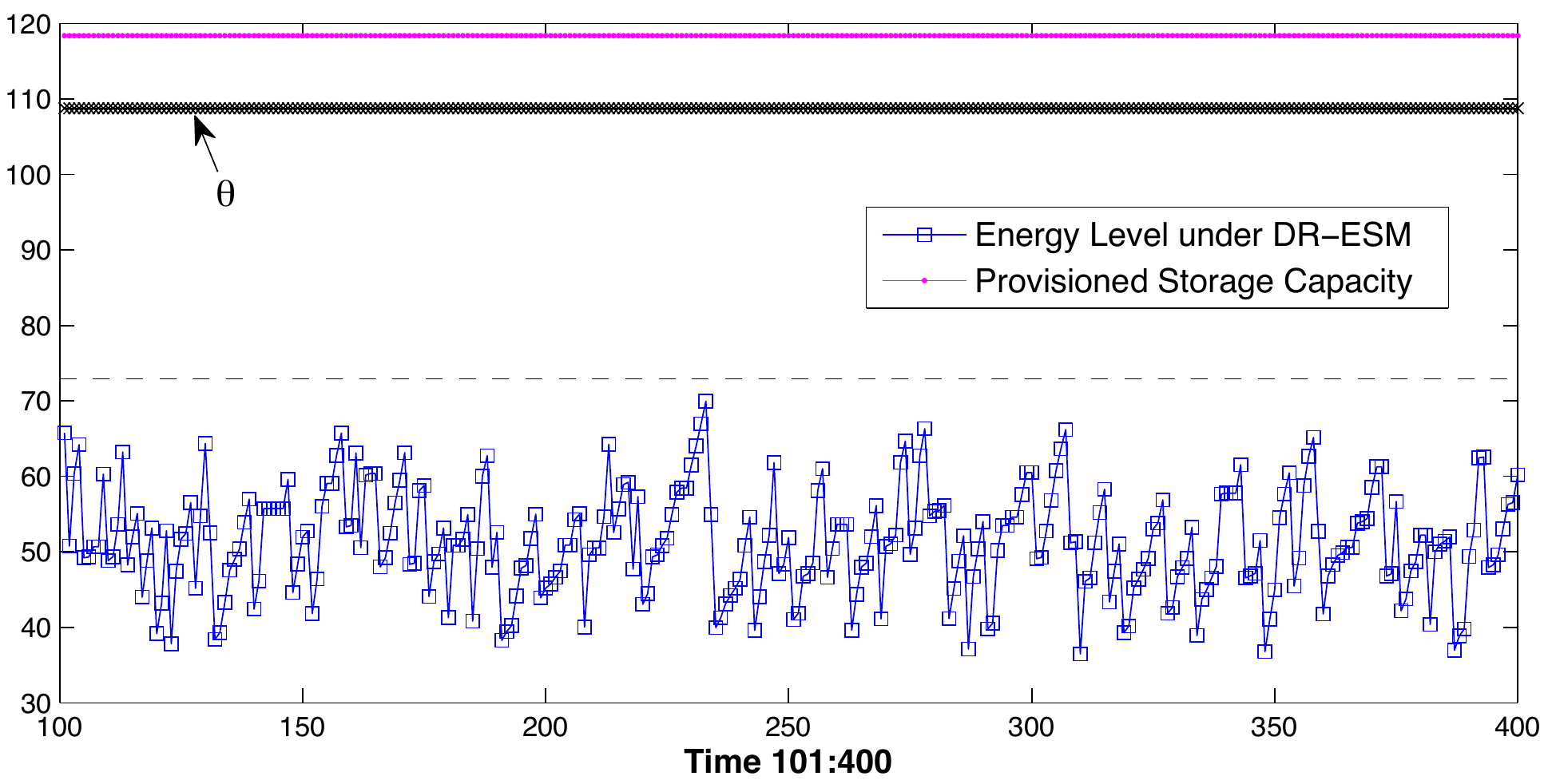}
\caption{A sample path energy level process under \textsf{DR-ESM} with $V=5$.}\label{fig:energy-level}
\end{figure}

\vspace{-.1in}
\section{Conclusion}\label{section:conclusion}
In this paper, we developed optimal energy management and demand response schemes for general power consuming systems with finite energy storage and renewable energy. Based on the Lyapunov optimization technique, we developed two very light-weight energy management schemes \textsf{ESM} and \textsf{DR-ESM} for load-serving and demand-response, respectively. Both schemes only require the user to solve a simple convex optimization program for decision making, and allow us to explicitly compute the required energy storage size. We proved that both schemes are able to achieve near-optimal performance.

\section*{Appendix A - Proof of Lemma \ref{lemma:basic}}
Here we prove Lemma \ref{lemma:basic}. 
\begin{proof} (Lemma \ref{lemma:basic})
Using (\ref{eq:storage-dyn}), we see that: 
\begin{eqnarray}
&& E(t+1)-\theta = E(t) -\theta \\
&& \qquad\quad  -\eta_e(d_s(t))+h_s(t)) + \eta_i(d_c(t)+r_c(t)).\nonumber
\end{eqnarray}
Square both sides of the above, we get:
\begin{eqnarray}
\hspace{-.3in}&&(E(t+1)-\theta)^2 \label{eq:energy-drift}\\
\hspace{-.3in}&&= (E(t)-\theta)^2 + [\eta_e(d_s(t))+h_s(t)) - \eta_i(d_c(t)+r_c(t))]^2\nonumber\\
\hspace{-.3in}&& \quad - 2(E(t)-\theta) [\eta_e(d_s(t))+h_s(t)) - \eta_i(d_c(t)+r_c(t))]. \nonumber
\end{eqnarray}
Now using (\ref{eq:charging-cond}) and (\ref{eq:discharging-cond}), we see that:  
\begin{eqnarray*}
[\eta_e(d_s(t))+h_s(t)) - \eta_i(d_c(t)+r_c(t))]^2\leq \eta_e^2c_{\textsf{dis}}^2 +  \eta_i^2c_{\textsf{char}}^2. 
\end{eqnarray*}
Thus, by defining $B\triangleq \frac{1}{2}[\eta_e^2c_{\textsf{dis}}^2 +  \eta_i^2c_{\textsf{char}}^2]$, multiplying both sides of (\ref{eq:energy-drift}) by $\frac{1}{2}$, and using the above, we get:
\begin{eqnarray*}
\hspace{-.3in}&&\frac{1}{2}[(E(t+1)-\theta)^2 - (E(t)-\theta)^2]\\
\hspace{-.3in}&&\quad\leq B - (E(t)-\theta )[\eta_e(d_s(t))+h_s(t)) - \eta_i(d_c(t)+r_c(t))]. 
\end{eqnarray*}
Taking expectations over the randomness of the  actions conditioning in $E(t)$, and using the definition of $\Delta(t)$, we prove Lemma \ref{lemma:basic}. 
\end{proof}

\vspace{-.1in}
\section*{Appendix B - Proof of (\ref{eq:cost-bound}) of Theorem \ref{theorem:per-esm}}
In this subsection we prove (\ref{eq:cost-bound}) of Theorem \ref{theorem:per-esm}. 
To do so, we first have the following theorem, which can be proven using a similar augment as in \cite{huangneely-spn-per}. 

\begin{theorem}\label{theorem:opt-stationary}
There exists a stationary and randomized energy management policy $\Pi$ that achieves the following: 
\begin{eqnarray}
\hspace{-.3in}&&\expect{f^{\Pi}(t)} = f_{\textsf{av}}^*, \label{eq:stat-per}\\
\hspace{-.3in}&& \expect{\eta_e (d^{\Pi}_s(t)+h^{\Pi}_s(t)) + \eta_i (d^{\Pi}_c(t)+r^{\Pi}_c(t))}=0. \label{eq:stat-queue}
\end{eqnarray}
Here the expectation is taken over the random system dynamics and the potential randomness of the charing/discharging and purchasing/selling actions. $\Box$
\end{theorem}

We remark that Theorem \ref{theorem:opt-stationary} holds for both the load-serving case as well as the demand response case. 
Note that although Theorem \ref{theorem:opt-stationary} shows that there exists such an optimal policy $\Pi$, it may not be implementable in practice. This is because finding it requires knowing all the statistical knowledge of the system parameters, including the prices, the loads,  and the renewable energy, and the required energy storage size is very difficult to compute (may also be infinite). In this case, our \textsf{ESM} algorithm provides a low-complexity alternative to achieve a similar performance. 

We now use Theorem \ref{theorem:opt-stationary} to prove  (\ref{eq:cost-bound}) of Theorem \ref{theorem:per-esm}. 
\begin{proof} ((\ref{eq:cost-bound}) of Theorem \ref{theorem:per-esm}) 
We recall (\ref{eq:esm-drift-foo1}) as follows: 
\begin{eqnarray}
\hspace{-.2in}&&\Delta(t) + V\expect{ f(t) \left.|\right. E(t)} \label{eq:esm-drift-foo2}\\
\hspace{-.2in}&&\leq B + V\expect{p(t) [d_l(t) +d_c(t)]  - q(t) h_s(t) \left.|\right. E(t)}\nonumber\\
\hspace{-.2in}&&\qquad -\, (E(t) -\theta ) \expect{  \eta_e (d_s(t) + h_s(t)) \nonumber\\
\hspace{-.2in}&& \qquad\qquad\qquad\qquad\qquad\quad  - \eta_i d_c(t) -\eta_ir_c(t) \left.|\right. E(t)}.\nonumber
\end{eqnarray}
Now since \textsf{ESM} is chosen to minimize the RHS of (\ref{eq:esm-drift-foo2}), the value of the RHS is no larger than that under the policy $\Pi$. Thus, (\ref{eq:esm-drift-foo2}) holds when we plug in the actions chosen by policy $\Pi$. This yields: 
\begin{eqnarray}
\hspace{-.2in}&&\Delta(t) + V\expect{ f^{\textsf{ESM}}(t) \left.|\right. E(t)} \label{eq:esm-drift-foo3}\\
\hspace{-.2in}&&\leq B + V\expect{p(t) [d^{\Pi}_l(t) +d^{\Pi}_c(t)]  - q(t) h^{\Pi}_s(t) \left.|\right. E(t)}\nonumber\\
\hspace{-.2in}&&\qquad -\, (E(t) -\theta ) \expect{  \eta_e (d^{\Pi}_s(t) + h^{\Pi}_s(t)) \nonumber\\
\hspace{-.2in}&& \qquad\qquad\qquad\qquad\qquad\quad  - \eta_i d^{\Pi}_c(t) -\eta_ir^{\Pi}_c(t) \left.|\right. E(t)}\nonumber\\
\hspace{-.2in}&&=B + Vf_{\textsf{av}}^*. 
\end{eqnarray}
Here in the last step we have used (\ref{eq:stat-per}) and (\ref{eq:stat-queue}) in Theorem \ref{theorem:opt-stationary}.  
Therefore, taking an expectation over $E(t)$ on both sides, and summing the above over $t=0, ..., T-1$, we get: 
\begin{eqnarray}
\expect{G(T) - G(0)} + \sum_{t=0}^{T-1}V\expect{ f^{\textsf{ESM}}(t)}  \leq TB + TVf_{\textsf{av}}^*. 
\end{eqnarray}
Rearranging the terms and dividing both sides by $TV$, we see that:
\begin{eqnarray}
\frac{1}{T} \sum_{t=0}^{T-1}\expect{ f^{\textsf{ESM}}(t)}  \leq f_{\textsf{av}}^* + \frac{B}{V}+\frac{\expect{G(0)}}{VT}. 
\end{eqnarray}
Taking a $\limsup$ as $T\rightarrow\infty$ and using the fact that $\expect{G(0)}<\infty$, we prove (\ref{eq:cost-bound}). 
\end{proof}

\vspace{-.1in}
\section*{Appendix - Proof  of Theorem \ref{theorem:per-esm-dr}}
In this section, we prove Theorem \ref{theorem:per-esm-dr}. The proof is very similar to the one of Theorem  \ref{theorem:per-esm-dr} except for the lower bound of (\ref{eq:energy-bound-dr}). Hence, we only present the proof of this part. 
\begin{proof} (Theorem \ref{theorem:per-esm-dr}) 
Let $\tilde{L}^*(t)$, $d_l^*(t)$, $d^*_c(t)$, $d^*_s(t)$, $h_s^*(t)$, $r^*_c(t)$ be an optimal solution of (\ref{eq:esm-dr}). 
We first notice that $\tilde{L}^*(t)$ will not directly affect the other actions. 
Hence, we can treat $\tilde{L}^*(t)$ as a given load. 

Now we prove the lower bound using induction. Assume the lower bound holds at time $t$. 
\begin{enumerate}
\item Suppose $E(t)\geq \eta_e\min[L_{\textsf{max}}, c_{\textsf{dis}}]$. 
Since  the maximum amount that can be discharged from the storage is $ \eta_e\min[L_{\textsf{max}}, c_{\textsf{dis}}]$, we see that $E(t+1)\geq0$.

\item  Suppose $E(t)< \eta_e\min[L_{\textsf{max}}, c_{\textsf{dis}}]$. In this case, we see from (\ref{eq:theta-def}) that:  
\begin{eqnarray}
\eta_e(E(t)-\theta), \eta_i(E(t)-\theta)<-  \frac{\max[p_{\textsf{max}}, q_{\textsf{max}}]}{\epsilon}.
\end{eqnarray}
We similarly see from (\ref{eq:theta-foo}) that:  
\begin{eqnarray}
W_h(t), W_l(t), W_c(t), W_r(t)<0. 
\end{eqnarray}
Note here we have $W_l(t)$ instead of $W_s(t)$. In this case, we similarly see that $h^*_s(t)=0$, 
and that either (i) $d_c^*(t)+r^*_c(t)=c_{\textsf{char}}$ or (ii) $d^*_l(t)+d_c^*(t)=c_{\textsf{grid}}$.

Suppose $d^*_l(t)+d_c^*(t)=c_{\textsf{grid}}$, then using (\ref{eq:lower-bound-foo}), we see that $E(t+1)\geq E(t)\geq0$. 
Else suppose $d_c^*(t)+r_c^*(t)=c_{\textsf{char}}$.  
From the objective function (\ref{eq:esm-1}) and the constraints (\ref{eq:esm-cond-1}) and (\ref{eq:esm-cond-2}), we again see that $d^*_l(t)$ should be as large as possible since $W_l(t)<0$ in (\ref{eq:esm-1}). Hence, we have: 
\begin{eqnarray}
d^*_l(t) = \min[c_{\textsf{grid}}- d^*_c(t), [L(t)]^+]. 
\end{eqnarray}
Using the proof of Theorem \ref{theorem:per-esm}, we see that this implies $E(t+1)\geq E(t)\geq0$. 
\end{enumerate}
The rest of the proof follows similarly as in the proof of Theorem \ref{theorem:per-esm}. 
\end{proof}

$\vspace{-.2in}$
\bibliographystyle{unsrt}
\bibliography{../grid-ref}

\end{document}